\documentclass[a4paper,12pt]{amsart}
\usepackage{latexsym,amsmath,amssymb}

\newtheorem{thm}{Theorem}[section]
\newtheorem{cor}[thm]{Corollary}
\newtheorem{lem}[thm]{Lemma}

\theoremstyle{definition}
\newtheorem{defn}[thm]{Definition}
\theoremstyle{remark}

\numberwithin{equation}{section}

\begin{document}
\title[]{\textsc{Essential norm of composition operators harmonic Bloch spaces }}
\author{\textsc{ Y. Estaremi, S. Esmaeili and A. Ebadian }}
\address{\textsc{ Y. estaremi, S. Esmaeili and A. Ebadian}} \email{estaremi@gmail.com},\email{dr.somaye.esmaili@gmail.com},\email{ ebadian.ali@gmail.com.}

\address{Department of mathematics, Payame Noor university , P. O. Box: 19395-3697, Tehran,
Iran.\\
}

\thanks{}
\thanks{}
\date{}
\dedicatory{}

\begin{abstract}
In this paper we characterize essential norm of composition operators on the spaces of Harmonic Bloch functions. These results extends the similar results that were proven for composition operators on Bloch spaces.
\end{abstract}

\maketitle

\commby{}

\section{\textsc{Introduction}}

 The essential norm $\|T\|_{e}$ of a continuous linear operator $T$ between Banach spaces $X$ and $Y$ is defined as the distance from $T$ to the space of compact operators from $X$ to $Y$.
 The essential norm of composition operators have been studied on analytic function spaces in \cite{sh,cp,dr,mor,mon,ch,mz,cm}.

 Let $D$ be the open unit disk in the complex plane. Let $\varphi$ be an analytic self-map of $D$, i. e., an analytic function $\varphi$ in $D$ such that $\varphi(D)\subset D$. The composition operator $C_{\varphi}$ induced by such $\varphi$ is the linear map on the spaces of all harmonic functions on the unit disk defined by
$$C_{\varphi}f=fo\varphi.$$

 The main goal of this paper is to compute the essential norm of $C_{\varphi}$ in terms of an asymptotic bound involving the quantity
$$\frac{(1-|z|^2)^\alpha}{(1-|\varphi(z)|^2)^\alpha}|\varphi^{'}(z)|.$$
Further, we obtain an other essential norm formula for composition operators on $HB(\alpha)$ for any $0<\alpha<\infty$ in terms of $\varphi^{n}$, where, $\varphi^{n}$ means the $n$-th power of $\varphi$.

Let $D$ be the open unit disk in the complex plane. For a continuously differentiable complex-valued $f(z)=u(z)+i\upsilon(z),$ $z=x+iy,$ we use the common notation for its formal derivatives:
$$f_{z}=\frac{1}{2}(f_{x}-if_{y}),$$  $$f_{\bar{z}}=\frac{1}{2}(f_{x}+if_{y}).$$

 A twice continuously differentiable complex-valued function $f=u+i\upsilon$ on $D$ is called a harmonic function if and only if the real-valued function $u$ and $\upsilon$ satisfy Laplace's equation $\Delta u=\Delta \upsilon=0$.

A direct calculation shows that the Laplacian of $f$ is
$$\Delta f=4f_{z\bar{z}}.$$
Thus for functions $f$ with continuous second partial derivatives, it is clear that $f$ is harmonic if ana only if $\Delta f=0.$
We consider complex-valued harmonic function $f$ defined in a simply connected domain $D\subset C.$ The function $f$ has a canonical decomposition $f=h+\bar{g},$ where $h$ and $g$ are analytic in $D$ \cite{dp}.
A planar complex-valued harmonic function $f$ in $D$ is called a harmonic Bloch function if and only if

$$\beta_{f}=\sup_{z,w\in D,z\neq w}\frac{|f(z)-f(w)|}{\varrho(z,w)}<\infty.$$

Here $\beta_{f}$ is the Lipschitz number of $f$ and

\begin{align*}
\varrho(z,w)=\arctan h|\frac{z-w}{1-\bar{z}w}|,
\end{align*}

denotes the hyperbolic distance between $z$ and $w$ in $D$, where here $\rho(z,w)$ is the pseudo-hyperbolic distance on $D$.
In \cite{cf} Colonna proved that
\begin{align*}
\beta_{f}=\sup_{z\in D}(1-|z|^2)[|f_{z}(z)|+|f_{\bar{z}}(z)|].
\end{align*}
Moreover, the set of all harmonic Bloch mappings, denoted by the symbol $HB(1)$ or $HB$, forms a complex Banach space with the norm $\|.\|$ given by
$$\|f\|_{HB(1)}=|f(0)|+\sup_{z\in D}(1-|z|^2)[|f_{z}(z)|+|f_{\bar{z}}(z)|].$$
Now we define the Harmonic $\alpha$-Bloch space $HB(\alpha)$.
 \begin{defn} For $\alpha\in(0,\infty)$, the Harmonic $\alpha$-Bloch space $HB(\alpha)$ consists of complex-valued harmonic function $f$ defined on $D$ such that
$$|||f|||_{HB(\alpha)}=\sup_{z\in D}(1-|z|^2)^\alpha[|f_{z}(z)|+|f_{\bar{z}}(z)|]<\infty,$$
and the harmonic little $\alpha$-Bloch space $HB_{0}(\alpha)$ consists of all function in $HB(\alpha)$ such that
$$\lim_{|z|\rightarrow1}(1-|z|^2)^\alpha[|f_{z}(z)|+|f_{\bar{z}}(z)|]=0.$$
\end{defn}
Obviously, when $\alpha=1$, we have $|||f|||_{HB(\alpha)}=\beta_{f}$. Each $HB(\alpha)$ is a Banach space with the norm given by
\begin{align*}
\|f\|_{HB(\alpha)}&=|f(0)|+\sup_{z\in D}(1-|z|^2)^\alpha[|f_{z}(z)|+|f_{\bar{z}}(z)|],
\end{align*}
and $HB_{0}(\alpha)$ is a closed subspace of $HB(\alpha)$.

%
%
\section{\textsc{Main results}}
In this section we characterize essential norm of composition operator $C_{\varphi}$ on $HB(\alpha)$.

Here we recall the next lemma from \cite{zr} that we need it in the sequel.
\begin{lem}\cite{zr}\label{l1} If $\alpha>0$,$n\in N$, $0<x<1$ and
$$H_{n,\alpha}(x)=x^{n-1}(1-x^{2})^{\alpha},$$
then $H_{n,\alpha}$ has the following properties:

a)
$$\max_{0\leq x\leq1}H_{n,\alpha}(x)=$$
$$H_{n,\alpha}(r_{n})=\left\{
   \begin{array}{ll}
     1 , & \hbox{n=1;} \\$$

     $$(\frac{2\alpha}{n-1+2\alpha})^{\alpha} $$
     $$(\frac{n-1}{n-1+2\alpha})^{\frac{(n-1)}{2}},$$
     $$& n\geq2.$$

    $$\end{array}
  \right.$$
In which
$$r_{n}=\left\{
   \begin{array}{ll}
     0 , & \hbox{n=1;} \\$$

     $$(\frac{n-1}{n-1+2\alpha})^{\frac{1}{2}},$$
     $$& n\geq2.$$

    $$\end{array}
  \right.$$

b) For $n\geq 1,H_{n,\alpha}$ is increasing on $[0,r_{n}]$ and decreasing on $[r_{n},1]$.\\

c) For $n\geq 1,H_{n,\alpha}$ is decreasing on $[r_{n},r_{n+1}],$ and so
$$\min_{x\in[r_{n},r_{n+1}]}H_{n,\alpha}(x)=H_{n,\alpha}(r_{n+1})=(\frac{2\alpha}{n+2\alpha})^{\alpha}(\frac{n}{n+2\alpha})^{\frac{(n-1)}{2}}.$$
Consequently,
$$\lim_{n\rightarrow\alpha}n^{\alpha}\min_{x\in[r_{n},r_{n+1}]}H_{n,\alpha}(x)=(\frac{2\alpha}{e})^{\alpha}.$$
\end{lem}
By using the Lemma \ref{l1} we can find a lower bound for the essential norm of $C_{\varphi}:HB(\alpha)\rightarrow HB(\alpha)$.
\begin{thm}\label{t12}
Let $0< \alpha < \infty$, $\varphi$ be an analytic self-map of the unit disk $D$ and $C_{\varphi}$ be a bounded operator from $HB(\alpha)$ into $HB(\alpha)$. Then
$$\|C_{\varphi}\|_{e}\geq \lim_{s\rightarrow 1}\sup _{|\varphi(z)|> s}\frac{(1-|z|^{2})^\alpha}{(1-|\varphi(z)|^{2})^\alpha} |\varphi^{'}(z)|.$$
\end{thm}
\begin{proof} Let $n\in \mathbb{N}$ and consider the function $z^{n}+\bar{z}^{n}$. By the Lemma \ref{l1} we have
\begin{align*}
\|z^{n}+\bar{z}^{n}\|_{HB(\alpha)}&=\max_{z\in D}2n |z|^{n-1}(1-|z|^2)^{\alpha}\\
&=2n(\frac{2\alpha}{n-1+2\alpha})^{\alpha}(\frac{n-1}{n-1+2\alpha})^{\frac{(n-1)}{2}},
\end{align*}
in which the maximum is attained at any point on the circle with radius $r_{n}=(\frac{n-1}{n-1+2\alpha})^{\frac{(n-1)}{2}}$.

Let$f_{n}=\frac{z^{n}+\bar{z}^{n}}{\|z^{n}+\bar{z}^{n}\|_{HB(\alpha)}}$. Then $f_{n}$ converges to $0$ weakly in $HB(\alpha)$. In particular, if $K$ is any compact operator on $HB(\alpha)$, then $\lim_{n\rightarrow\infty}\|Kf_{n}\|_{HB(\alpha)}=0$. For
$$D_{n}=\{z\in D : r_{n}\leq |\varphi(z)|\leq r_{n+1}\},$$
we have
\begin{align*}
\min_{z\in D_{n}}(1-|z|^2)^{\alpha}[|(f_{n})_{z}(z)|+|(f_{n})_{\bar{z}}(z)|]&=(1-|r_{n+1}|^{2})^{\alpha}[|h_{n}'(z)|+|g_{n}'(z)|]\\
&=(\frac{n-1+2\alpha}{n+2\alpha})^{\alpha}(\frac{n^{2}(2\alpha-1)n}{n^{2}(2\alpha-1)n-2\alpha}).
\end{align*}
It is easy to see that this minimum tends to $1$ when $n\rightarrow\infty$. Also for each $n\geq2$ the minimum  is attained at any point of the circle centered at the origin and of radius $r_{n+1}$. For any compact operator $K$ on $HB(\alpha)$, we get that
\begin{align*}
\|C_{\varphi}-K\|&\geq \limsup_{n\rightarrow\infty}\|(C_{\varphi}-K)f_{n}\|_{HB(\alpha)}\\
&\geq\limsup_{n\rightarrow\infty}(\|(C_{\varphi})f_{n}\|_{HB(\alpha)}-\|(Kf_{n}\|_{HB(\alpha)})\\
&=\limsup_{n\rightarrow\infty}\|(C_{\varphi})f_{n}\|_{HB(\alpha)}.
\end{align*}
Thus we obtain the followings:
\begin{align*}
\|C_{\varphi}\|_{e}&=\inf_{K}\|C_{\varphi}-K\|\\
&\geq\limsup_{n\rightarrow\infty}\|(C_{\varphi})f_{n}\|_{HB(\alpha)}\\
&=\limsup_{n\rightarrow\infty}\sup_{z\in D}\frac{(1-|z|^2)^\alpha}{(1-|\varphi(z)|^2)^\alpha}|\varphi^{'}(z)|(1-|\varphi(z)|^2)^\alpha [|h_{n}^{'}(\varphi(z))|+|g_{n}^{'}(\varphi(z))|]\\
&\geq\lim_{n\rightarrow\infty}\sup_{\varphi(z)\in D_{n}}\frac{(1-|z|^2)^\alpha}{(1-|\varphi(z)|^2)^\alpha}|\varphi^{'}(z)|\min_{\varphi(z)\in D_{n}}(1-|\varphi(z)|^2)^\alpha [|h_{n}^{'}(\varphi(z))|+|g_{n}^{'}(\varphi(z))|]\\
&\geq\lim_{n\rightarrow\infty}\sup_{\varphi(z)\in D_{n}}\frac{(1-|z|^2)^\alpha}{(1-|\varphi(z)|^2)^\alpha}|\varphi'(z)|\min_{w\in D_{n}}(1-|w|^2)^\alpha [|h_{n}^{'}(w)|+|g_{n}^{'}(w)|],
\end{align*}
Since
$$\limsup_{n\rightarrow\infty}\min_{w \in D_{n}}(1-|w|^2)^\alpha [|h_{n}^{'}(w)|+|g_{n}^{'}(w)|]=1,$$
then we conclude that
$$\|C_{\varphi}\|_{e}\geq \lim_{n\rightarrow\infty}\sup_{\varphi(z)\in D_{n}}\frac{(1-|z|^2)^\alpha}{(1-|\varphi(z)|^2)^\alpha}|\varphi^{'}(z)|.$$
\end{proof}

We need some Lemmas to obtain some upper bounds for the essential norm of $C_{\varphi}$. For $r\in(0,1)$, let $K_{r}f(z)=f(rz)$. Then by Theorems 2. 6 and 2.7 of \cite{eee} we get that $K_{r}$ is a compact operator on the spaces $HB(\alpha)$ and $HB_{0}(\alpha)$ for any positive number $\alpha$, with $\|K_{r}\|\leq1$.

\begin{lem}\label{l2}
 For $0<\alpha<1$, there exists a sequence $\{r_{k}\}$, $0<r_{k}<1$, tending to $1$, such that the compact operator $L_{n}=(\frac{1}{n})\sum_{k=1}^{n}K_{r_{k}}$ on $HB_{0}(\alpha)$ satisfies:

 a) For every $t\in[0,1)$ we have
 $$\lim_{n\rightarrow\infty} \sup_{\|f\|_{HB(\alpha)}\leq1} \sup_{|z|\leq t}[|((I-L_{n})f)_{z}(z)|+|((I-L_{n})f)_{\bar{z}}(z)|]=0.$$

b) $$\lim_{n\rightarrow\infty} \sup_{\|f\|_{HB(\alpha)}\leq1} \sup_{z\in D}|(I-L_{n})f(z)|=0.$$

c) $$\limsup_{n\rightarrow\infty}\|I-L_{n}\|\leq1.$$
\end{lem}
\begin{proof}
In order to prove the Lemma we need to prove that, for any $t$, $0<t<1$, and any $\varepsilon>0$, there is an $n>0$ such that, for any $n>N,$

$$\sup_{\|f\|_{HB(\alpha)}\leq1} \sup_{|z|\leq t}[|((I-L_{n})f)_{z}(z)|+|((I-L_{n})f)_{\bar{z}}(z)|]<\varepsilon,$$
$$\sup_{\|f\|_{HB(\alpha)}\leq1} \sup_{z\in D}|(I-L_{n})f(z)|<\varepsilon,$$
and
$$\|I-L_{n}\|\leq1+2\varepsilon.$$
Let $S_{HB_{0}(\alpha)}$ denote the unit ball of $HB_{0}(\alpha)$ and the positive number $s_{1}$ such that $t<s_{1}<1$. Since $S_{HB_{0}(\alpha)}$ is a relatively compact subset with respect to the topology $\tau$ of the uniform convergence on compact subsets of the unit disk and $(I-K_{r})f$ tends to $0$ with respect to $\tau$, then $((I-K_{r})f)^{'}$ tends to $0$. And so there is $0<r_{1}<1$ such that
\begin{align*}
&\sup_{\|f\|_{HB(\alpha)}\leq1} \sup_{|z|\leq s_{1}}(1-|z|^2)^\alpha[|((I-K_{r_{1}})f)_{z}(z)|+|((I-K_{r_{1}})f)_{\bar{z}}(z)|]\\
&<\min\{\varepsilon,\varepsilon(1-|s_{1}|^2)^\alpha\}.
\end{align*}
Since $K_{r_{1}}$ is compact, then $K_{r_{1}}S_{HB_{0}(\alpha)}$ is relatively compact and we can find $s_{2}>s_{1}$ such that

$$\sup_{\|f\|_{HB(\alpha)}\leq1} \sup_{|z|> s_{2}}(1-|z|^2)^\alpha[|(K_{r_{1}}f)_{z}(z)|+|(K_{r_{1}}f)_{\bar{z}}(z)|]<\varepsilon.$$

And so we get that
$$\sup_{\|f\|_{HB(\alpha)}\leq1} \sup_{|z|> s_{2}}(1-|z|^2)^\alpha[|((I-K_{r_{1}})f)_{z}(z)|+|((I-K_{r_{1}})f)_{\bar{z}}(z)|]<1+\varepsilon.$$

By repeating this method we can find the sequences $\{s_{k}\}$ and $\{r_{k}\}$ satisfying
\begin{align*}
&\sup_{\|f\|_{HB(\alpha)}\leq1} \sup_{|z|\leq s_{k}}(1-|z|^2)^\alpha[|((I-K_{r_{k}})f)_{z}(z)|+|((I-K_{r_{k}})f)_{\bar{z}}(z)|]\\
&<\min\{\varepsilon,\varepsilon(1-|s_{1}|^2)^\alpha\}
\end{align*}
and
$$\sup_{\|f\|_{HB(\alpha)}\leq1} \sup_{|z|> s_{k+1}}(1-|z|^2)^\alpha[|((I-K_{r_{k}})f)_{z}(z)|+|((I-K_{r_{k}})f)_{\bar{z}}(z)|]<1+\varepsilon.$$
If we take $n>\frac{1}{\varepsilon}$ and set $L_{n}=(\frac{1}{n})\sum_{k=1}^{n}K_{r_{k}}$, the by using the fact that $(1-|s_{1}|^2)^\alpha\}\leq(1-|z|^2)^\alpha$ for $|z|\leq s_{1}$, we have
\begin{align*}
&\sup_{\|f\|_{HB(\alpha)}\leq1} \sup_{|z|\leq t}[|((I-L_{n})f)_{z}(z)|+|((I-L_{n})f)_{\bar{z}}(z)|]\\
&\leq\frac{1}{n}\sum_{k=1}^{n}\sup_{\|f\|_{HB(\alpha)}\leq1} \sup_{|z|\leq s_{1}}[|((I-K_{r_{k}})f)_{z}(z)|+|((I-K_{r_{k}})f)_{\bar{z}}(z)|]\\
&<\varepsilon.
\end{align*}
This gives (a). Now we prove (c). We know that for $n\geq1$ and $0<\alpha<1$,
\begin{align*}
\|I-L_{n}\|&=\sup_{\|f\|_{HB(\alpha)}\leq1}\|(I-L_{n})f\|_{HB(\alpha)}\\
&=\sup_{\|f\|_{HB(\alpha)}\leq1} \sup_{|z|\leq t}(1-|z|^2)^\alpha[|((I-L_{n})f)_{z}(z)|+|((I-L_{n})f)_{\bar{z}}(z)|]\\
&+\sup_{\|f\|_{HB(\alpha)}\leq1} \sup_{|z|>t}(1-|z|^2)^\alpha[|((I-L_{n})f)_{z}(z)|+|((I-L_{n})f)_{\bar{z}}(z)|]\\
&=I_{1}+I_{2}.
\end{align*}
In which
$$I_{1}=\sup_{\|f\|_{HB(\alpha)}\leq1} \sup_{|z|\leq t}(1-|z|^2)^\alpha[|((I-L_{n})f)_{z}(z)|+|((I-L_{n})f)_{\bar{z}}(z)|]<\varepsilon$$
and
$$I_{2}=\sup_{\|f\|_{HB(\alpha)}\leq1} \sup_{s_{l}<|z|<s_{l+1},k\neq l}(1-|z|^2)^\alpha[|((I-L_{n})f)_{z}(z)|+|((I-L_{n})f)_{\bar{z}}(z)|]+$$
$$\sup_{\|f\|_{HB(\alpha)}\leq1} \sup_{|z|>s_{k+1}}(1-|z|^2)^\alpha[|((I-L_{n})f)_{z}(z)|+|((I-L_{n})f)_{\bar{z}}(z)|].$$
For $k=l$, $s_{l}<|z|<s_{l+1}$ and $\|f\|\leq1$, we have
$$(1-|z|^2)^\alpha[|((I-k_{r_{l}})f)_{z}(z)|+|((I-k_{r_{l}})f)_{\bar{z}}(z)|]\leq\|I-k_{r_{l}}\|\leq\|I\|+\|k_{r_{l}}\|\leq2,$$
except possibly for $k\neq l$. However, in this case we get that
$$(1-|z|^2)^\alpha[|((I-k_{r_{l}})f)_{z}(z)|+|((I-k_{r_{l}})f)_{\bar{z}}(z)|]\leq1+\varepsilon$$
and so
\begin{align*}
&(1-|z|^2)^\alpha[|((I-L_{n})f)_{z}(z)|+|((I-L_{n})f)_{\bar{z}}(z)|]\\
&\leq\frac{1}{n}\sum_{k\neq l}(1-|z|^2)^\alpha[|((I-k_{r_{k}})f)_{z}(z)|+|((I-k_{r_{k}})f)_{\bar{z}}(z)|]\\
&+\frac{1}{n}(1-|z|^2)^\alpha[|((I-k_{r_{l}})f)_{z}(z)|+|((I-k_{r_{l}})f)_{\bar{z}}(z)|]\\
&<\frac{n-1}{n}(1+\varepsilon)+\frac{2}{n}\\
&<1+2\varepsilon.
\end{align*}
Therefore we get (c). The proof of the property (b) is exactly as in the proof of the Lemma 1 of \cite{mon}.
\end{proof}
Similar to the Lemma \ref{l2} for the case $\alpha=1$ we have the next Lemma.

\begin{lem}\label{l3} There is a sequence $\{r_{k}\}$, $0<r_{k}<1$, tending to 1, such that the compact operator $L_{n}=(\frac{1}{n})\sum_{k=1}^{n}K_{r_{k}}$ on $HB_{0}(\alpha)$ satisfies:

a) For any $t\in[0,1)$ we have
$$\lim_{n\rightarrow\infty} \sup_{\|f\|_{HB(\alpha)}\leq1} \sup_{|z|\leq t}[|((I-L_{n})f)_{z}(z)|+|((I-L_{n})f)_{\bar{z}}(z)|]=0.$$

b) $\lim_{n\rightarrow\infty} \sup_{\|f\|_{HB(\alpha)}\leq1} \sup_{|z|>s}|(I-L_{n})f(z)|(-\log(1-|z|^2))^{-1}\leq1$ for $s$ sufficiently close to 1
and
$$\lim_{n\rightarrow\infty} \sup_{\|f\|_{HB(\alpha)}\leq1} \sup_{|z|\leq s}|(I-L_{n})f(z)|=0.$$
c) $\limsup_{n\rightarrow\infty}\|I-L_{n}\|\leq1$.
\end{lem}
\begin{proof} The proof of (a) and (b) is exactly the same as the proof of Lemma \ref{l2}. Also the proof of (b) is given in given in \cite{mz}.
\end{proof}
Also for the case $\alpha>1$, we have the following Lemma.
\begin{lem}\label{l4}
For $\alpha>1$, there is a sequence $\{r_{k}\}$ with $0<r_{k}<1$, tending to $1$. Also, the compact operator $L_{n}=(\frac{1}{n})\sum_{k=1}^{n}K_{r_{k}}$ on $HB_{0}(\alpha)$ has the followings properties:

a) For any $s\in[0,1)$, $$\lim_{n\rightarrow\infty} \sup_{\|f\|_{HB(\alpha)}\leq1} \sup_{|z|\leq s}[|((I-L_{n})f)_{z}(z)|+|((I-L_{n})f)_{\bar{z}}(z)|]=0.$$

b) For any $t\in[0,1)$, $$\lim_{n\rightarrow\infty} \sup_{\|f\|_{HB(\alpha)}\leq1} \sup_{|z|\leq t}|(I-L_{n})f(z)|=0.$$

c) $\limsup_{n\rightarrow\infty}\|I-L_{n}\|\leq1$.
\end{lem}
\begin{proof}
The proof of (a) and (c) is the same as the proof of the Lemma \ref{l2}. Hence we prove (b). For harmonic function $f=h+\bar{g}$ we have
\begin{align*}
|(I-L_{n})f(z)|&\leq|(I-L_{n})h(z)|+|(I-L_{n})g(z)|\\
&\leq\int_{D}\frac{(1-|w|^2)^{1+\alpha}}{|w||1-\bar{z}w|^{2+\alpha}}[|((I-L_{n})h)^{'}(z)|+|((I-L_{n})g)^{'}(z)|]dA(w)\\
&=\int_{D}\frac{(1-|w|^2)^{1+\alpha}}{|w||1-\bar{z}w|^{2+\alpha}}[|((I-L_{n})f)_{w}(z)|+|((I-L_{n})f)_{w}(z)|]dA(w).
\end{align*}
For each $t\in[0,1)$ and any $\varepsilon>0$ we can find $s\in[0,1)$ sufficiently close to $1$ such that
$$\frac{1-s}{s(1-t)^{2+\alpha}}<\varepsilon.$$
Hence we get that
\begin{align*}
I_{1}&=\sup_{\|f\|_{HB(\alpha)}\leq1} \sup_{|z|\leq t}\int_{D\setminus D_{s}}\frac{(1-|w|^2)^{1+\alpha}}{|w||1-\bar{z}w|^{2+\alpha}}[|((I-L_{n})f)_{w}(z)|+|((I-L_{n})f)_{w}(z)|]dA(w)\\
&\leq\sup_{\|f\|_{HB(\alpha)}\leq1}\|(I-L_{n})f\|_{HB(\alpha)}\sup_{|z|\leq t}|\int_{D\setminus D_{s}}\frac{(1-|w|^2)^{1+\alpha}}{|w||1-\bar{z}w|^{2+\alpha}}dA(w)\\
&\leq\frac{1}{s(1-t)^{2+\alpha}}(1-s)<\varepsilon,
\end{align*}
 when $s$ is sufficiently close to $1$. By (a) we can find $N>0$ such that for all $n>N$,
$$\sup_{\|f\|_{HB(\alpha)}\leq1} \sup_{|z|\leq s}[|((I-L_{n})f)_{z}(z)|+|((I-L_{n})f)_{\bar{z}}(z)|]<\varepsilon,$$
and so we conclude that
\begin{align*}
I_{2}&=\sup_{\|f\|_{HB(\alpha)}\leq1} \sup_{|z|\leq t}\int_{D_{s}}\frac{(1-|w|^2)^{1+\alpha}}{|w||1-\bar{z}w|^{2+\alpha}}[|((I-L_{n})f)_{w}(z)|+|((I-L_{n})f)_{w}(z)|]dA(w)\\
&\leq\frac{\varepsilon}{(1-t)^{2+\alpha}}\sup_{|z|\leq t}\int_{D_{s}}\frac{(1-|w|^2)^{1+\alpha}}{|w|}dA(w)<C\varepsilon,
\end{align*}
for some $C>0$. Therefore by combining the above inequalities, we get that for $n>N$,
\begin{align*}
&\sup_{\|f\|_{HB(\alpha)}\leq1} \sup_{|z|\leq t}|(I-L_{n})f(z)|\\
&\leq\sup_{\|f\|_{HB(\alpha)}\leq1} \sup_{|z|\leq t}\int_{D}\frac{(1-|w|^2)^{1+\alpha}}{|w||1-\bar{z}w|^{2+\alpha}}[|((I-L_{n})f)_{w}(z)|+|((I-L_{n})f)_{w}(z)|]dA(w)\\
&\leq I_{1}+I_{2}<(1+C)\varepsilon.
\end{align*}
This implies that
$$\lim_{n\rightarrow\infty} \sup_{\|f\|_{HB(\alpha)}\leq1} \sup_{|z|\leq t}|(I-L_{n})f(z)|=0.$$
\end{proof}
Now we can find an upper bound for the essential norm of $C_{\varphi}:HB(\alpha)\rightarrow HB(\alpha)$.
\begin{thm}\label{t14}
Let $\alpha>0$ and $C_{\varphi}:HB(\alpha)\rightarrow HB(\alpha)$ be  bounded. Then we have

$$\|C_{\varphi}\|_{e}\leq \lim_{s\rightarrow 1}\sup _{|\varphi(z)|> s}\frac{(1-|z|^{2})^\alpha}{(1-|\varphi(z)|^{2})^\alpha} |\varphi^{'}(z)|.$$
\end{thm}
\begin{proof} Let $L_{n}$ be a sequence of operators given in Lemma \ref{l4}. Since each $L_{n}$ is a compact operator from $HB(\alpha)$ to $HB(\alpha)$, then $C_{\varphi}L_{n}$ is compact and we have
\begin{align*}
\|C_{\varphi}\|_{e}&\leq\|C_{\varphi}-C_{\varphi}L_{n}\|\\
&=\|C_{\varphi}(I-L_{n})\|\\
&=\sup_{\|f\|_{HB(\alpha)}\leq1}\|C_{\varphi}(I-L_{n})f\|_{HB(\alpha)}\\
&\leq\sup_{\|f\|_{HB(\alpha)}\leq1}|(I-L_{n})f(\varphi(0))|\\
&+\sup_{\|f\|_{HB(\alpha)}\leq1}\sup_{z\in D}(1-|z|^2)^\alpha |\varphi^{'}(z)|[|((I-L_{n})f)_{z}(\varphi(z))|+[|((I-L_{n})f)_{\bar{z}}(\varphi(z))|].
\end{align*}
By the Lemma \ref{l4} we can suppose the term $\sup_{\|f\|_{HB(\alpha)}\leq1}|(I-L_{n})f(\varphi(0))|$ sufficiently small for some large enough $n$.
Now we need only consider the term
$$\sup_{\|f\|_{HB(\alpha)}\leq1}\sup_{z\in D}(1-|z|^2)^\alpha |\varphi^{'}(z)|[|((I-L_{n})f)_{z}(\varphi(z))|+[|((I-L_{n})f)_{\bar{z}}(\varphi(z))|].$$
For every $0<s<1$, we bound this expression from above by
\begin{align*}
&\sup_{\|f\|_{HB(\alpha)}\leq1}\sup_{|\varphi(z)|\leq s}(1-|z|^2)^\alpha |\varphi^{'}(z)|[|((I-L_{n})f)_{z}(\varphi(z))|+[|((I-L_{n})f)_{\bar{z}}(\varphi(z))|]\\
&+\sup_{\|f\|_{HB(\alpha)}\leq1}\sup_{|\varphi(z)|> s}(1-|z|^2)^\alpha |\varphi^{'}(z)|[|((I-L_{n})f)_{z}(\varphi(z))|+[|((I-L_{n})f)_{\bar{z}}(\varphi(z))|].
\end{align*}
Since $C_{\varphi}$ is bounded from $HB(\alpha)$ into $HB(\alpha)$, then we conclude that
$$(1-|z|^2)^\alpha |\varphi^{'}(z)|\leq (1-|\varphi(z)|^{2})^{\alpha}\sup_{z\in D}(1-|z|^2)^\alpha |\varphi^{'}(z)|<\infty,$$
and so
$$\sup_{|\varphi(z)|\leq s}(1-|z|^2)^\alpha |\varphi^{'}(z)|<\infty.$$
Hence by the Lemma \ref{l4} we see that
$$\lim_{n\rightarrow\infty}\sup_{\|f\|_{HB(\alpha)}\leq1}\sup_{|\varphi(z)|\leq s}(1-|z|^2)^\alpha |\varphi^{'}(z)|[|((I-L_{n})f)_{z}(\varphi(z))|+[|((I-L_{n})f)_{\bar{z}}(\varphi(z))|]=0.$$
Easily we find that the term
$$\sup_{\|f\|_{HB(\alpha)}\leq1}\sup_{|\varphi(z)|> s}\frac{(1-|z|^2)^\alpha |\varphi^{'}(z)|}{(1-|\varphi(z)|^{2})^{\alpha}}(1-|\varphi(z)|^{2})^{\alpha}[|((I-L_{n})f)_{z}(\varphi(z))|+[|((I-L_{n})f)_{\bar{z}}(\varphi(z))|]$$
is bounded above by
\begin{align*}
&\sup_{\|f\|_{HB(\alpha)}\leq1}\|(I-L_{n})f\|_{HB(\alpha)}\sup_{|\varphi(z)|> s}\frac{(1-|z|^2)^\alpha |\varphi^{'}(z)|}{(1-|\varphi(z)|^{2})^{\alpha}}\\
&=\|I-L_{n}\|\sup_{|\varphi(z)|> s}\frac{(1-|z|^2)^\alpha }{(1-|\varphi(z)|^{2})^{\alpha}}|\varphi^{'}(z)|.
\end{align*}
Therefore we have
\begin{align*}
&\limsup_{n\rightarrow\infty}\sup_{\|f\|_{HB(\alpha)}\leq1}\sup_{|\varphi(z)|> s}(1-|z|^2)^\alpha |\varphi^{'}(z)|[|((I-L_{n})f)_{z}(\varphi(z))|+[|((I-L_{n})f)_{\bar{z}}(\varphi(z))|]\\
&\leq\sup_{|\varphi(z)|> s}\frac{(1-|z|^2)^\alpha }{(1-|\varphi(z)|^{2})^{\alpha}}|\varphi^{'}(z)|,
\end{align*}
and so we get that
$$\|C_{\varphi}\|_{e}\leq\sup_{|\varphi(z)|> s}\frac{(1-|z|^2)^\alpha }{(1-|\varphi(z)|^{2})^{\alpha}}|\varphi^{'}(z)|.$$
Consequently by taking the limit as $s\rightarrow1$, we get the result.
\end{proof}
We remark that a similar result can be obtained for the essential norm of $C_{\varphi}$ acting boundedly on $HB_{0}(\alpha)$, for $\alpha>0$, where in the essential norm formula in the Theorem \ref{t14} we may replace $\lim_{s\rightarrow 1}\sup _{|\varphi(z)|> s}$ by $\limsup_{|z|\rightarrow1}$. This equivalence will be shown in the following Theorem:
\begin{thm}\label{t17}
If $\alpha>0$ and $C_{\varphi}:HB_{0}(\alpha)\rightarrow HB_{0}(\alpha)$ is  bounded. Then
$$\|C_{\varphi}\|_{e}=\limsup_{|z|\rightarrow1}\frac{(1-|z|^{2})^\alpha}{(1-|\varphi(z)|^{2})^\alpha} |\varphi^{'}(z)|.$$
\end{thm}
\begin{proof}
By Theorems \ref{t14}, it would be sufficient to prove the equality
$$\limsup_{|z|\rightarrow1}\frac{(1-|z|^{2})^\alpha}{(1-|\varphi(z)|^{2})^\alpha} |\varphi^{'}(z)|=\lim_{s\rightarrow 1}\sup _{|\varphi(z)|> s}\frac{(1-|z|^{2})^\alpha}{(1-|\varphi(z)|^{2})^\alpha} |\varphi^{'}(z)|.$$
This equivalence can be shown by a similar argument of Theorem 2.2 in \cite{mor} and using the characterization of boundedness of $C_{\varphi}:HB_{0}(\alpha)\rightarrow HB_{0}(\alpha)$ for $\alpha>0$.
\end{proof}
Recall that the composition operator $C_{\varphi}:HB(\alpha)\rightarrow HB(\alpha)$ is said to be bounded below, if there exists $\varepsilon>0$ such that
$$|||C_{\varphi}f|||_{HB(\alpha)}\geq\varepsilon|||f|||_{HB(\alpha)},$$
for any $f\in HB(\alpha)$. In the next theorem we characterize boundedness of below of $C_{\varphi}$ on $HB(\alpha)$.
\begin{thm}\label{lem11}
The composition operator $C_{\varphi}$ is bounded below on $HB(\alpha)$ if and only if there exists $\delta>0$ such that $\|C_{\varphi}(f)\|_{HB(\alpha)}\geq \delta \|f\|_{HB(\alpha)}$ for $f\in HB(\alpha)$.
\end{thm}
\begin{proof}
First we assume that $C_{\varphi}$ is bounded below on $HB(\alpha)$, so there exists a $\delta>0$ such that $|||go\varphi|||_{HB(\alpha)}\geq \delta |||g|||_{HB(\alpha)}$ for $g\in HB(\alpha)$. Let $g(z)=f(z)-f(\varphi(0))$ for $f\in HB(\alpha)$. Hence $g\in HB(\alpha)$, $g(\varphi(0))=0$, $\|g\|_{HB(\alpha)}=\|f\|_{HB(\alpha)}$ and $\|go\varphi\|_{HB(\alpha)}=\|fo\varphi\|_{HB(\alpha)}$. Therefore
\begin{align*}
\|C_{\varphi}(f)\|_{HB(\alpha)}&=\|fo\varphi\|_{HB(\alpha)}\\
&=|||go\varphi|||_{HB(\alpha)}\\
&\geq \delta \|g\|_{HB(\alpha)}=\delta \|f\|_{HB(\alpha)}.
\end{align*}
Conversely, suppose that there exists $\delta>0$ such that $\|C_{\varphi}(f)\|_{HB(\alpha)}\geq \delta \|f\|_{HB(\alpha}$ for all $f\in HB(\alpha)$. Then for every $f\in HB(\alpha)$, we have
\begin{align*}
|f(0)|&\leq |f(\varphi(0))|+\int_{o}^{\varphi(0)}[|(f)_{z}(z)|+|(f)_{\bar{z}}(z)|]|dz|\\
&=|f(\varphi(0))|+\int_{o}^{\varphi(0)}(1-|z|^2)^{\alpha}[|(f)_{z}(z)|+|(f)_{\bar{z}}(z)|]\frac{|dz|}{(1-|z|^2)^{\alpha}}\\
&\leq |f(\varphi(0))|+\|f\|_{HB(\alpha)}\int_{o}^{\varphi(0)}\frac{|dz|}{(1-|z|^2)^{\alpha}}.
\end{align*}
If $\alpha=1$, then
$$|f(0)|\leq |f(\varphi(0))|+\frac{\|f\|_{HB(\alpha)}}{2}\ln \frac{1+|\varphi(0)|}{1-|\varphi(0)|}$$
and so
\begin{align*}
|||f|||_{HB(\alpha)}&=|f(0)|+\|f\|_{HB(\alpha)}\\
&\leq\frac{|||fo\varphi|||_{HB(\alpha)}}{\delta}(1+\frac{1}{2}ln \frac{1+|\varphi(0)|}{1-|\varphi(0)|}).
\end{align*}
Since, $\|fo\varphi\|_{HB(\alpha)}\geq \delta \|f\|_{HB(\alpha)}$.
As the same way if $\alpha\neq1$,
\begin{align*}
|f(0)|&\leq |f(\varphi(0))|+\|f\|_{HB(\alpha)}\int_{o}^{\varphi(0)}\frac{|dz|}{(1-|z|^2)^{\alpha}}\\
&\leq|f(\varphi(0))|+\|f\|_{HB(\alpha)}\int_{o}^{\varphi(0)}\frac{|dz|}{(1-|z|)^{\alpha}},
\end{align*}
so
\begin{align*}
|f(0)|&\leq |f(\varphi(0))|+\frac{\|f\|_{HB(\alpha)}}{-\alpha+1}[(1-|\varphi(0)|)^{-\alpha+1}-1]|||f|||_{HB(\alpha)}\\
&\leq \frac{|||fo\varphi|||_{HB(\alpha)}}{\delta}(1+\frac{1}{-\alpha+1}[(1-|\varphi(0)|)^{-\alpha+1}-1]).
\end{align*}
By these observations we get that $C_{\varphi}$ is bounded below. \cite{ch}
\end{proof}
Here we give an equivalent condition for boundedness of $C_{\varphi}:HB(\alpha)\rightarrow HB(\alpha)$.
\begin{thm}\label{t15}
Let $0<\alpha<\infty$ and $\varphi$ be an analytic self-map of the unit disk $D$. Then $C_{\varphi}:HB(\alpha)\rightarrow HB(\alpha)$ is bounded if and only if
$$\sup_{n\in N}n^{\alpha-1}\|\varphi^{n}+\bar{\varphi}^{n}\|_{HB(\alpha)}<\infty.$$
\end{thm}
\begin{proof}
Let $n\in N$ and $z^{n}+\bar{z}^{n}$. By Theorem \ref{lem11} we have
\begin{align*}
\|z^{n}+\bar{z}^{n}\|_{HB(\alpha)}&=\max_{z\in D}2n |z|^{n-1}(1-|z|^2)^{\alpha}\\
&=2n(\frac{2\alpha}{n-1+2\alpha})^{\alpha}(\frac{n-1}{n-1+2\alpha})^{\frac{(n-1)}{2}},
\end{align*}
where the maximum is attained at any point on the circle with radius $r_{n}=(\frac{n-1}{n-1+2\alpha})^{\frac{(n-1)}{2}}$. Therefore
$$\lim_{n\rightarrow\infty}n^{\alpha-1}\|z^{n}+\bar{z}^{n}\|_{HB(\alpha)}=2(\frac{2\alpha}{e}).$$
 And so there exists a constant $K>0$, independent of$n$, such that $\|z^{n}+\bar{z}^{n}\|_{HB(\alpha)}\leq K n^{1-\alpha}$. For $f_{n}=\frac{z^{n}+\bar{z}^{n}}{\|z^{n}+\bar{z}^{n}\|_{HB(\alpha)}}$ we get that $|||f_{n}|||_{HB(\alpha)}=\|f_{n}\|_{HB(\alpha)}=1$. Hence we have
 \begin{align*}
 \infty&>\|C_{\varphi}\|\\
 &\geq\|C_{\varphi}f_{n}\|_{HB(\alpha)}\\
 &=\|f_{n}o\varphi\|_{HB(\alpha)}\\
 &\geq\frac{1}{K}n^{\alpha-1}\|\varphi^{n}+\bar{\varphi}^{n}\|_{HB(\alpha)}.
 \end{align*}
Conversely, by the assumption obviously we have $\|\varphi+\bar{\varphi}\|_{HB(\alpha)}<\infty$.  If $\sup_{z\in D}|\varphi(z)|<1$, then we can find $0<r<1$ such that $\sup_{z\in D}|\varphi(z)|<r$. In this case easily we get that $C_{\varphi}:HB(\alpha)\rightarrow HB(\alpha)$ is bounded. Now we assume that $\sup_{z\in D}|\varphi(z)|=1$ and put for any $n\geq1$
$$D_{n}=\{z\in D : r_{n}\leq |\varphi(z)|\leq r_{n+1}\},$$
where $r_{n}$ is given as before. Let $m$ be the smallest positive integer such that $D_{m}\neq\varnothing$. Since $\sup_{z\in D}|\varphi(z)|=1$, $D_{n}$ is not empty for every integer $n\geq m$, and $D=\bigcup_{n=m}^{\infty}D_{n}$. Then for every $n\geq m$ we conclude that
\begin{align*}
\min_{z\in D_{n}}2 n^{\alpha}|\varphi(z)|^{n-1}(1-|\varphi(z)|^{2})^{\alpha}&\geq2 n^{\alpha}H_{n,\alpha}(r_{n+1})\\
&=2(\frac{2\alpha n}{n+2\alpha})^{\alpha}(\frac{n}{n+2\alpha})^{\frac{(n-1)}{2}},
\end{align*}
where $H_{n,\alpha}(x)=x^{n-1}(1-x^2)^{\alpha}$. Thus
$$\lim_{n\rightarrow\infty}\min_{z\in D_{n}}2 n^{\alpha}|\varphi(z)|^{n-1}(1-|\varphi(z)|^{2})^{\alpha}\geq2(\frac{2\alpha}{e}).$$
Therefore, there exists $\delta>0$ such that, for any $n\geq m,$
$$\min_{z\in D_{n}}2 n^{\alpha}|\varphi(z)|^{n-1}(1-|\varphi(z)|^{2})^{\alpha}\geq\delta.$$
Also for every $f\in HB(\alpha)$ we have
\begin{align*}
\|C_{\varphi}f\|_{HB(\alpha)}&=\sup_{z\in D}(1-|z|^2)^\alpha |\varphi^{'}(z)|[|f_{z}(z)|+|f_{\bar{z}}(z)|]\\
&=\sup_{n\geq m}\sup_{z\in D_{n}}(1-|z|^2)^\alpha |\varphi^{'}(z)|[|f_{z}(z)|+|f_{\bar{z}}(z)|]\\
&\leq\frac{1}{\delta}\sup_{n\geq m}\sup_{z\in D_{n}}n^{\alpha-1}(1-|\varphi(z)|^{2})^{\alpha}[|f_{z}(z)|+|f_{\bar{z}}(z)|]
2|(\varphi^{n})^{'}(z)|(1-|z|^2)^\alpha\\
&\leq\frac{1}{\delta}\|f\|_{HB(\alpha)}\sup_{n\geq1}n^{\alpha-1}\|\varphi^{n}+\bar{\varphi}^{n}\|_{HB(\alpha)}.
\end{align*}
This implies that $C_{\varphi}$ is bounded on $HB(\alpha)$.
\end{proof}
It is well-known that the composition operator $C_{\varphi}$ is bounded on the Bloch space $HB$ for any analytic self-map $\varphi$ of $D$. Hence, if we let $\alpha=1$ in Theorem \ref{t15}, we get the following corollary.
\begin{cor}
For any analytic self-map $\varphi$ of the unit disk $D$,
$$\sup_{n\in N}\|\varphi^{n}+\bar{\varphi}^{n}\|_{HB}<\infty.$$
\end{cor}
\begin{thm}\label{t16}
Let $0< \alpha < \infty$ and $\varphi$ be an analytic self-map of the unit disk $D$. Then the essential norm of the composition operator $C_{\varphi}:HB(\alpha)\rightarrow HB(\alpha)$ is
$$\|C_{\varphi}\|_{e}=\frac{1}{2}(\frac{e}{2\alpha})^{\alpha}\limsup_{n\rightarrow\infty}n^{\alpha-1}\|\varphi^{n}+\bar{\varphi}^{n}\|_{HB(\alpha)}.$$
\end{thm}
\begin{proof}
Let $n\in N$ and $z^{n}+\bar{z}^{n}$. By Theorem \ref{t15} we have
\begin{align*}
\|z^{n}+\bar{z}^{n}\|_{HB(\alpha)}&=\max_{z\in D}2n |z|^{n-1}(1-|z|^2)^{\alpha}\\
&=2n(\frac{2\alpha}{n-1+2\alpha})^{\alpha}(\frac{n-1}{n-1+2\alpha})^{\frac{(n-1)}{2}},
\end{align*}
where the maximum is attained at any point on the circle with radius
$$r_{n}=(\frac{n-1}{n-1+2\alpha})^{\frac{(n-1)}{2}}.$$
Thus we have
$$\lim_{n\rightarrow\infty}n^{\alpha-1}\|z^{n}+\bar{z}^{n}\|_{HB(\alpha)}=2(\frac{2\alpha}{e}).$$
 Let $f_{n}=\frac{z^{n}+\bar{z}^{n}}{\|z^{n}+\bar{z}^{n}\|_{HB(\alpha)}}$ and so $\|f_{n}\|_{HB(\alpha)}=1$. Also easily we get that $\{f_{n}\}$ converges to $0$ weakly in $HB(\alpha)$. If $K$ is a compact operator on $HB(\alpha)$, then $\lim_{n\rightarrow\infty}\|Kf_{n}\|_{HB(\alpha)}=0$.
 Therefore
 \begin{align*}
 \|C_{\varphi}-K\|&\geq \limsup_{n\rightarrow\infty}\|(C_{\varphi}-K)f_{n}\|_{HB(\alpha)}\\
 &\geq\limsup_{n\rightarrow\infty}(\|(C_{\varphi})f_{n}\|_{HB(\alpha)}-\|(Kf_{n}\|_{HB(\alpha)})\\
 &=\limsup_{n\rightarrow\infty}\|(C_{\varphi})f_{n}\|_{HB(\alpha)}.
 \end{align*}
 By taking the $\inf$ on both sides of this inequality over all compact operator $K$, for $C_{\varphi}:HB(\alpha)\rightarrow HB(\alpha)$, we obtain the that
 \begin{align*}
 \|C_{\varphi}\|_{e}&=\inf_{K}\|C_{\varphi}-K\|\\
 &\geq\limsup_{n\rightarrow\infty}\|(C_{\varphi})f_{n}\|_{HB(\alpha)}\\
 &=\frac{1}{2}(\frac{e}{2\alpha})^{\alpha}\limsup_{n\rightarrow\infty}n^{\alpha-1}\|\varphi^{n}+\bar{\varphi}^{n}\|_{HB(\alpha)}.
 \end{align*}
  If
  $$\limsup_{n\rightarrow\infty}n^{\alpha-1}\|\varphi^{n}+\bar{\varphi}^{n}\|_{HB(\alpha)}=\infty,$$
   then $C_{\varphi}$ is not bounded on $HB(\alpha)$. Now suppose that
   $$\limsup_{n\rightarrow\infty}n^{\alpha-1}\|\varphi^{n}+\bar{\varphi}^{n}\|_{HB(\alpha)}<\infty.$$
Hence $C_{\varphi}$ is a bounded operator from $HB(\alpha)$ to $HB(\alpha)$ and, clearly,
$\|\varphi+\bar{\varphi}\|_{HB(\alpha)}<\infty$.
Notice that if $\sup_{Z\in D}|\varphi(z)|<1$, then $C_{\varphi}$ is compact and we have the conditions of theorem. Hence, in the sequel we assume that $\sup_{Z\in D}|\varphi(z)|=1$. Let $L_{n}$ be a sequence of operators as we used in the last results. Since each $L_{n}$ is compact as an operator on $HB(\alpha)$, so is $C_{\varphi}L_{n}$. And we have
\begin{align*}
\|C_{\varphi}\|_{e}&=\inf_{C_{\varphi}L_{n} is compact}\|C_{\varphi}-C_{\varphi}L_{n}\|\\
&\leq\limsup_{n\rightarrow\infty}\|C_{\varphi}-C_{\varphi}L_{n}\|\\
&\leq\limsup_{n\rightarrow\infty}\sup_{\|f\|_{HB(\alpha)}\leq1}|(I-L_{n})f(\varphi(0))|\\
&+\limsup_{n\rightarrow\infty}\sup_{\|f\|_{HB(\alpha)}\leq1}\sup_{z\in D}(1-|z|^2)^\alpha |\varphi^{'}(z)|
[|((I-L_{n})f)_{z}(\varphi(z))|+[|((I-L_{n})f)_{\bar{z}}(\varphi(z))|].
\end{align*}
Consequently, we have $$\limsup_{n\rightarrow\infty}\sup_{\|f\|_{HB(\alpha)}\leq1}|(I-L_{n})f(\varphi(0))|=0.$$
Now we need only consider the term
$$J=\sup_{\|f\|_{HB(\alpha)}\leq1}\sup_{z\in D}(1-|z|^2)^\alpha |\varphi^{'}(z)|
[|((I-L_{n})f)_{z}(\varphi(z))|+[|((I-L_{n})f)_{\bar{z}}(\varphi(z))|].$$
For every $n\geq1$ put
$$D_{n}=\{z\in D : r_{n}\leq |\varphi(z)|\leq r_{n+1}\},$$
where $r_{n}$ is as given in the Lemma 2.16. Let $m$ be the smallest positive integer such that $D_{m}\neq\varnothing$. Since $\sup_{z\in D}|\varphi(z)|=1$, $D_{n}$ is not empty for every integer $n\geq m$ and $D=\bigcup_{n=m}^{\infty}D_{n}$. Now we divide $J$ into two parts:
\begin{align*}
J&=\sup_{\|f\|_{HB(\alpha)}\leq1}\sup_{m\leq k\leq N-1}\sup_{z\in D_{k}}(1-|z|^2)^\alpha |\varphi^{'}(z)|[|((I-L_{n})f)_{z}(\varphi(z))|
+[|((I-L_{n})f)_{\bar{z}}(\varphi(z))|]\\
&+\sup_{\|f\|_{HB(\alpha)}\leq1}\sup_{k\geq N}\sup_{z\in D_{k}}(1-|z|^2)^\alpha |\varphi^{'}(z)|
[|((I-L_{n})f)_{z}(\varphi(z))|+[|((I-L_{n})f)_{\bar{z}}(\varphi(z))|]\\
&=J_{1}+J_{2}.
\end{align*}
For $J_{1}$ we have
\begin{align*}
\limsup_{n\rightarrow\infty}J_{1}&=\limsup_{n\rightarrow\infty}\sup_{\|f\|_{HB(\alpha)}\leq1}\sup_{m\leq k\leq N-1}\sup_{z\in D_{k}}(1-|z|^2)^\alpha
|\varphi^{'}(z)|[|((I-L_{n})f)_{z}(\varphi(z))|\\
&+[|((I-L_{n})f)_{\bar{z}}(\varphi(z))|]\\
&=\limsup_{n\rightarrow\infty}\sup_{\|f\|_{HB(\alpha)}\leq1}\sup_{r_{m}\leq|\varphi(z)|\leq r_{N-1}}[|((I-L_{n})f)_{z}(\varphi(z))|+[|((I-L_{n})f)_{\bar{z}}(\varphi(z))|]=0.
\end{align*}
Write the function under three supremum signs in $J_{2}$ as
$$\frac{(1-|z|^2)^\alpha |\varphi^{'}(z)|
[|((I-L_{n})f)_{z}(\varphi(z))|+[|((I-L_{n})f)_{\bar{z}}(\varphi(z))|k^{\alpha}|\varphi(z)|^{k-1}(1-|\varphi(z)|^{2})^{\alpha}}{k^{\alpha}|\varphi(z)|^{k-1}(1-|\varphi(z)|^{2})^{\alpha}}.$$
For $z\in D_{k}$,
$$2k^{\alpha}|\varphi(z)|^{k-1}(1-|\varphi(z)|^{2})^{\alpha}\geq2(\frac{2\alpha k}{k+2\alpha})^{\alpha}(\frac{k}{k+2\alpha})^{\frac{(k-1)}{2}}.$$
It is easy to see that
$$\lim_{n\rightarrow\infty}[2(\frac{2\alpha k}{k+2\alpha})^{\alpha}(\frac{k}{k+2\alpha})^{\frac{(k-1)}{2}}]^{-1}=\frac{1}{2}(\frac{e}{2\alpha})^{\alpha}.$$
Hence for every $\varepsilon>0$, we can choose $N>m+1$ large enough such that for any $k\geq N$,
$$[2(\frac{2\alpha k}{k+2\alpha})^{\alpha}(\frac{k}{k+2\alpha})^{\frac{(k-1)}{2}}]^{-1}<\frac{1}{2}(\frac{e}{2\alpha})^{\alpha}+\varepsilon.$$
For such $N$ we have
\begin{align*}
J_{2}&\leq[\frac{1}{2}(\frac{e}{2\alpha})^{\alpha}+\varepsilon]\sup_{\|f\|_{HB(\alpha)}\leq1}\|(I-L_{n})f\|_{HB(\alpha)}\sup_{k\geq N}\sup_{z\in D_{k}}2k^{\alpha-1}|(\varphi^{k})^{'}(z)|(1-|z|^2)^\alpha\\
&\leq[\frac{1}{2}(\frac{e}{2\alpha})^{\alpha}+\varepsilon]\|I-L_{n}\|\sup_{k\geq N}k^{\alpha-1}\|\varphi^{k}+\bar{\varphi}^{k}\|_{HB(\alpha)}.
\end{align*}
Thus by the Theorem \ref{t15} we obtain
$$\limsup_{k\rightarrow\infty}J_{2}\leq[\frac{1}{2}(\frac{e}{2\alpha})^{\alpha}+\varepsilon]\sup_{k\geq N}k^{\alpha-1}\|\varphi^{k}+\bar{\varphi}^{k}\|_{HB(\alpha)}.$$
Hence for any $N$ sufficiently large we have
\begin{align*}
\|C_{\varphi}\|_{e}\leq\limsup_{n\rightarrow\infty}J_{2}\leq[\frac{1}{2}(\frac{e}{2\alpha})^{\alpha}+\varepsilon]\sup_{k\geq N}k^{\alpha-1}\|\varphi^{k}+\bar{\varphi}^{k}\|_{HB(\alpha)}.
\end{align*}
Thus we have
$$\|C_{\varphi}\|_{e}\leq[\frac{1}{2}(\frac{e}{2\alpha})^{\alpha}+\varepsilon]\limsup_{k\rightarrow\infty}k^{\alpha-1}\|\varphi^{k}+\bar{\varphi}^{k}\|_{HB(\alpha)}.$$
Since $\varepsilon$ is an arbitrary positive number, therefore we get
$$\|C_{\varphi}\|_{e}\leq\frac{1}{2}(\frac{e}{2\alpha})^{\alpha}\limsup_{k\rightarrow\infty}k^{\alpha-1}\|\varphi^{k}+\bar{\varphi}^{k}\|_{HB(\alpha)}.$$
This completes the proof.
\end{proof}

\end{document}